\documentclass[11pt]{amsart}
\usepackage[latin1]{inputenc}
\usepackage{amssymb}
\usepackage{pdfsync}
\usepackage[english]{babel}

\usepackage[pdftex,pagebackref,colorlinks=true,urlcolor=blue,linkcolor=blue,citecolor=blue]{hyperref}
\usepackage{graphicx}

\usepackage[capitalize]{cleveref}
\usepackage{fullpage}
\usepackage{color}
\usepackage{amsmath}
\usepackage{amsfonts}
\usepackage{mathrsfs}
\usepackage{t1enc , graphicx}
\usepackage{verbatim}
\usepackage{bbm}

\usepackage[left= 1.5 in, right= 1.5 in ,top= 1 in, bottom = 2 in]{geometry}
\newcommand{\R}{\mathbb{R}}

\newcommand{\C}{\mathbb{C}}
\newcommand{\T}{\mathbb{T}}
\newcommand{\Q}{\mathbb{Q}}
\newcommand{\Z}{\mathbb{Z}}
\newcommand{\N}{\mathbb{N}}

\newtheorem{theorem}{Theorem}[section]
\newtheorem*{theorem*}{Theorem}
\newtheorem{proposition}[theorem]{Proposition}
\newtheorem{lemma}[theorem]{Lemma}

\newtheorem{corollary}[theorem]{Corollary}
\newtheorem*{corollary*}{Corollary}
\newtheorem{question}[theorem]{Question}
\newtheorem{conjecture}[theorem]{Conjecture}

\newtheorem{innercustomprop}{Proposition}
\newenvironment{customprop}[1]
  {\renewcommand\theinnercustomprop{#1}\innercustomprop}
  {\endinnercustomprop}

\theoremstyle{definition}
\newtheorem*{definition*}{Definition}
\newtheorem{definition}[theorem]{Definition}

\theoremstyle{remark}

\newtheorem{example}[theorem]{Example}
\newtheorem*{remark*}{Remark}
\newtheorem{innercustomexample}{Example}
\newenvironment{customexample}[1]
  {\renewcommand\theinnercustomexample{#1}\innercustomexample}
  {\endinnercustomexample}

\begin{document}
\author{Anh N. Le}
\address{Department of Mathematics\\
  Northwestern University\\
  2033 Sheridan Road, Evanston, IL 60208-2730, USA}
  
\email{anhle@math.northwestern.edu}

\title{Interpolation sets and nilsequences}
\maketitle

\begin{abstract}
    To give positive answer to a question of Frantzikinakis, we study a class of subsets of $\mathbb{N}$, called interpolation sets, on which every bounded sequence can be extended to an almost periodic sequence on $\mathbb{N}$. Strzelecki has proved that lacunary sets are interpolation sets. We prove that sets that are denser than all lacunary sets cannot be interpolation sets. We also extend the notion of interpolation sets to nilsequences and show that the analogue to Frantzikinakis' question for arbitrary sequences is false.
\end{abstract}

\section{Introduction}

\subsection{Motivation}
    \label{sec:motivation}
    Let $(X, \mathcal{B}, \mu, T)$ be a measure preserving system and $f_0, f_1, \ldots, f_k$ bounded functions on $X$. The sequence $a(n) = \int_X f_0(x) f_1(T^n x) \ldots f_k(T^{kn}x) \, d\mu(x)$ is called a \emph{$k$-multiple correlation}. Multiple correlations arise implicitly in Furstenberg's proof of Szemer\'edi's theorem, in which all $f_i$ are indicator functions of a set with positive measure \cite{Furstenberg77}. They are then defined and studied formally by Bergelson, Host and Kra in \cite{Bergelson_Host_Kra05}. Through what is now called the correspondence principle, these sequences capture the intersection of some translates of dense subsets of integers.
    
    To understand the structure of multiple correlations, Bergelson, Host and Kra \cite{Bergelson_Host_Kra05} introduced the notion of nilsequences. These sequences are obtained by evaluating continuous functions along the orbits in homogeneous spaces of nilpotent Lie groups (see \cref{sec:preliminary_nilsequences} for the precise definition). Among other things, they prove: Given a $k$-multiple correlation $(a(n))_{n \in \N}$ defined via an ergodic system $(X, \mathcal{B}, \mu, T)$, there exists a uniform limit of $k$-step nilsequences $(\psi(n))_{n \in \N}$ such that
    \begin{equation}
    \label{eq:mar-19-1}
        \lim_{N - M \to \infty} \frac{1}{N-M} \sum_{n = M}^{N-1} |a(n) - \psi(n)| = 0  
    \end{equation}
    This result is extended to general measure preserving systems (not necessarily ergodic) by Leibman \cite{Leibman15}. Inspired by these results, Frantzikinakis asks the following question:
    \begin{question}{\cite[Question 13]{frant17}} 
    \label{ques:mar-14-1}
        Let $r_n = p_n$ (n-th prime), $\lfloor n^c \rfloor$ for some $c > 0$, or $2^n$. Is it true that for any $k$-multiple correlation $(a(n))_{n \in \N}$, there exists a uniform limit of $k$-step nilsequence $(\psi(n))_{n \in \N}$ such that 
        \begin{equation*}
            \lim_{N \to \infty} \frac{1}{N} \sum_{n=1}^{N} |a(r_n) - \psi(r_n)| = 0?
        \end{equation*}
    \end{question}
    In \cite{Le17}, we give affirmative answers for $r_n = p_n$ and $\lfloor n^c \rfloor$ (The case of the primes $p_n$ is also proved by Tao and Ter\"av\"ainen \cite{tao-teravainen17}). Positive answers for these two sequences are expected because they share important properties with the full sequence of natural numbers $r_n = n$. 
    More specifically, we can use the Host-Kra Structure Theorem \cite{Host_Kra05,Ziegler07} to project the averages of multiple correlations along these two sequences to the nilfactors without affecting the averages. Furthermore, the orbits of totally ergodic nilrotations along these sequences are equidistributed on the entire nilmanifolds (see \cite{Le17} for details).
    
    On the other hand, it is easy to see that $(2^n)_{n \in \N}$ does not have these properties (see \cref{subsec:2^n}). Therefore the answer to \cref{ques:mar-14-1} for $(2^n)_{n \in \N}$, if affirmative, requires a different technique. Indeed we are able to show that one can obtain any bounded sequence by evaluating an almost periodic sequence (uniform limit of $1$-step nilsequences) along $(2^n)_{n \in \N}$, making the positive answer for this sequence vacuous. 
    
    Motivated by this result, we proceed to look for all sequences sharing this property with $(2^n)_{n \in \N}$. However, it turns out that this seemingly ergodic theoretical question was answered by harmonic analysts in the 1960's, and sequences with this property have been extensively studied under the name of interpolation sets \cite{Strzelecki_1963, Hartman_Ryll-Nardzewski_1964, Ryll-Nardzewski_1964, Kunen_Rudin_1999, Graham_Hare_2013}.
    
    Before going further, we remark that for a multiple correlation $(a(n))_{n \in \N}$, there exists a unique uniform limit of nilsequences $(\psi_a(n))_{n \in \N}$ satisfying \eqref{eq:mar-19-1}. Hence there is a question related to \cref{ques:mar-14-1}:  Is it true that $\lim_{N \to \infty} \frac{1}{N} \sum_{n=1}^N |a(r_n) - \psi_a(r_n)| = 0$ for $r_n = p_n, \lfloor n^c \rfloor$ or $2^n$? This question is answered affirmatively for $(p_n)_{n \in \N}$ and $(\lfloor n^c \rfloor)_{n \in \N}$ in \cite{Le17}. However, it is negative for $(2^n)_{n \in \N}$, which is in contrast with the answer for \cref{ques:mar-14-1}.
    
    \subsection{Interpolation sets for almost periodic sequences}
    \begin{definition*}
        A set $E = \{r_n\}_{n\in \N} \subset \N$ with $r_1 < r_2 < \ldots$ is called an \emph{$I_0$ set} (or \emph{interpolation set}) if for every bounded sequence $(b(n))_{n \in \N}$, there exists an almost periodic sequence $(\psi(n))_{n \in \N}$ such that $\psi(r_n) = b(n)$ for all $n \in \N$.
        
        In other word, every bounded sequence on $E$ can be extended to an almost periodic sequence on $\N$. 
    \end{definition*}
    
    The set $\{2^n\}_{n \in \N}$ is $I_0$ and this is a corollary of a result for lacunary sets. A set $E = \{r_n: n \in \N\} \subset \N$ with $r_1 < r_2 <\ldots$ is called \emph{lacunary} (or \emph{Hadamard}) if $\inf_{n \in \N} r_{n+1}/r_n > 1$ for all $n \in \N$. Strzelecki \cite{Strzelecki_1963} proved that lacunary sets are $I_0$. Hartman-Ryll-Nardzewski characterization states that $E$ is $I_0$ if and only if disjoint subsets of $E$ have disjoint closures in the Bohr compactification of $\Z$ \cite{Hartman_Ryll-Nardzewski_1964}. Equivalently, for every $A, B \subset E$ disjoint, the difference $A-B$ is not a set of Bohr recurrence (see \cref{sec:sets_of_bohr_recurrence} for the definition of sets of Bohr recurrence). This gives examples of $I_0$ sets that are not lacunary but finite unions of lacunary sets, such as $\{2^n\}_{n \in \N} \cup \{2^n + 1\}_{n \in \N}$. Grow \cite{Grow_1987} constructed a class of $I_0$ sets which are not finite unions of lacunary sets, for example, $\{3^{n^2} + 3^j: n \geq 1, (n-1)^2 \leq j \leq n^2\}$ (see also M\'ela \cite{Mela_1969}). 
     
    All examples of $I_0$ sets above are very sparse; in particular, they are all derived from lacunary sets. One may wonder if there exists an $I_0$ set that has polynomial growth? The only result in this direction seems to be by Hartman \cite{Hartman_1961} in which he observes that the sets $\{n^k\}_{n \in \N}$ with $k \in \N$ cannot be $I_0$ using Weyl equidistribution. In this paper, we fill in the gap by showing the following: A set $E = \{r_n\}_{n \in \N} \subset \N$ with $r_1 < r_2 < \ldots$ is called \emph{denser than lacunaries} if $\lim_{n \to \infty} r_n/s_n = 0$ for every lacunary set $\{s_n\}_{n \in \N}$ with $s_1 < s_2 < \ldots$.
    In \cref{sec:denser_than_lacunaries}, we prove:
    \begin{theorem}
    \label{thm:may-8-1}
        Sets that are denser than lacunaries are not $I_0$.
    \end{theorem}
    This together with Strzelecki's Theorem \cite{Strzelecki_1963} almost characterize all $I_0$ sets. For the sets that do not fall into these two categories, their behavior seems to depend on some delicate arithmetic property rather than just density, as following example points out:
    \begin{example}
    \label{example:may-9-1}
        The set $\{2^n\}_{n \in \N} \cup \{2^n + 2n - 1\}_{n \in \N}$ is  $I_0$, but $\{2^n\}_{n \in \N} \cup \{2^n + 2n\}_{n \in \N}$ is not $I_0$.
    \end{example}
    This example is an immediate corollary of Hartman-Ryll-Nardzewski characterization and Strzelecki's Theorem. As our method and motivation are different, we also prove it in \cref{sec:some_properties}.
    
    A desired property for $I_0$ sets is that they are stable under union with finite sets. In an effort to prove this, we show:
    \begin{proposition}
    \label{prop:may-8-2}
        A set of Bohr recurrence can be partitioned into two sets of Bohr recurrence.
    \end{proposition}
    As a corollary, we have a different proof of Ryll-Nardzewski's result.
    \begin{theorem}[Ryll-Nardzewski \cite{Ryll-Nardzewski_1964}, see also M\'ela \cite{Mela_1968}, Ramsey \cite{Ramsey_1980}]
        \label{thm:may-8-3}
        The union of an $I_0$ set with a finite set is $I_0$.
    \end{theorem}
    \cref{prop:may-8-2} and \cref{thm:may-8-3} are proved in \cref{sec:union_with_finite_set}.
    
    \subsection{Interpolation sets for nilsequences}
    
    Almost periodic sequences can be described as uniform limits of $1$-step nilsequences (see \cref{sec:preliminary_nilsequences} for the definition). Therefore, a natural question is to what extent the notion of $I_0$ sets extend to nilsequences? In \cref{sec:nilsequences}, we begin a study of this question. 
    \begin{definition*}
        For $k \in \N$, a set $E = \{r_n\}_{n\in \N} \subset \N$ with $r_1 < r_2 < \ldots$ is called a \emph{$k$-step-$I_0$ set} if for every bounded sequence $(b(n))_{n \in \N}$, there exists a uniform limit of $k$-step nilsequences $(\psi(n))_{n \in \N}$ such that $\psi(r_n) = b(n)$ for all $n \in \N$.
    \end{definition*}
    Similar to Hartman-Ryll-Nardzewski characterization, a set $E$ is $k$-step-$I_0$ if and only if any two disjoint subsets $A$ and $B$ of $E$ are separable by some $k$-step nilrotation (see \cref{sec:nessary-sufficient-nilsequences}). However, since nilrotations are not isometries in general, it cannot be concluded that $A - B$ is not a set of nil-recurrence. Therefore most of the results for almost periodic sequences do not easily carry over to nilsequences.
    
    To side step this problem, we look at a different aspect of nilsequences: averages. Inherited from the equidistribution property of nilrotations, nilsequences have nice properties with respect to averaging. For example, by Leibman \cite{leib05}, for a nilsequence $(\psi(n))_{n \in \N}$ and a polynomial $P \in \mathbb{Q}[n]$ taking integer values on $\Z$, the uniform average
    \begin{equation}
    \label{eq:mar-20-1}
        \lim_{N-M \to \infty} \frac{1}{N-M} \sum_{n=M}^{N-1} \psi(P(n))
    \end{equation}
    exists. Since not all bounded sequences have uniform averages, \eqref{eq:mar-20-1} shows that every polynomial set is not $k$-step-$I_0$ for any $k \in \N$. In \cref{sec:dense_subset_of_polynomial}, we extend this result to  positively dense subsets of polynomial sets using a lemma of Moreira, Richter and Robertson \cite{moreira-richter-robertson}. More specifically,
    \begin{definition*}
        Let $F = \{s_{n_i}\}_{i \in \N} \subset E = \{s_n\}_{n \in \N} \subset \N$. The upper density of $F$ relative to $E$ is defined to be
        \[
            \bar{d}_E(F) = \limsup_{N - M \to \infty} \frac{|\{n_i: i \in \N\} \cap \{M + 1, M + 2, \ldots, N\}|}{N - M}
        \]
    \end{definition*}
    \begin{theorem}
    \label{prop:jan-12-1}
        Let $E = \{P(n)\}_{n \in \N} \subset \N$ where $P \in \Q[n]$ non-constant and taking integer values on $\Z$. Then every subset of $E$ with positive relative upper density is not $k$-step-$I_0$ for any $k \in \N$. 
    \end{theorem}
    
    Intuitively, one expects that there is more freedom when increasing the step $k$ in $k$-step nilsequences. This intuition is partially confirmed by following proposition, which is proved in \cref{sec:2_step}: 
    \begin{proposition}
    \label{prop:apr-27-1}
        There exists a set that is $2$-step-$I_0$ but not $1$-step-$I_0$.
    \end{proposition}
    More generally, we conjecture that
    \begin{conjecture}
        For every $k \in \N$, there exists a set that is $(k+1)$-step-$I_0$ but not $k$-step-$I_0$.
    \end{conjecture}
    On the other hand, motivated by \cref{thm:may-8-1}, we ask:
    \begin{question}    
        Is it true that every set that is denser than lacunaries is not $k$-step-$I_0$ for any $k \in \N$?
    \end{question}
    The answer is probably negative. 
    
    \subsection{The analogue to Frantzikinakis's question for arbitrary sequences does not hold}
    It is shown in \cite{Le17} that the answer to \cref{ques:mar-14-1} is affirmative for all $r_n = P(n)$ or $P(p_n)$ where $P \in \Z[n]$ non-constant, and for $(r_n)_{n \in \N}$ in a large class of Hardy field sequences, for example, $r_n = \lfloor n \log n \rfloor$ or $\lfloor n^2 \sqrt{2} + n \sqrt{3} \rfloor$. By Strzelecki's Theorem and Hartman-Ryll-Nardzewski characterization, it still holds for lacunary and some unions of lacunary sequences. Hence one may wonder if there exists a sequence for which the answer to \cref{ques:mar-14-1} is negative? We provide such example in  \cref{sec:negative_to_frantzikinakis_question} by showing that
    \begin{proposition}
    \label{prop:negative_to_frantzikiankis_question}
    There exists an increasing sequence of natural numbers $(r_n)_{n \in \N}$ and a  $1$-correlation $(a(n))_{n \in \N}$ such that for any almost periodic sequence $(\psi(n))_{n \in \N}$,
    \[
        \liminf_{N \to \infty} \sum_{n=1}^N |a(r_n) - \psi(r_n)| > 0.
    \]
\end{proposition}

\subsection{Acknowledgment}
   We are grateful for Bryna Kra for her guidance during the course of this project. We thank Joel Moreira and Florian Richter for many helpful discussions and John Griesmer for pointing out the connection between our project and results in harmonic analysis. We also thank Kathryn Hare for answering our question regarding interpolation sets. 
\section{Preliminaries}
\label{sec:background}

\subsection{Notation}
    For $N \in \N$, let $[N]$ denote the set $\{1, 2, \ldots, N\}$ and $\T$  denote the torus $\R/\Z$. The notation $(b(n))_{n \in \N}$ is used for a sequence of complex numbers and $\{r_n\}_{n \in \N}$ for a subset of $\N$ with $r_1 < r_2 < \ldots$
\subsection{Nilsequences}
\label{sec:preliminary_nilsequences}
    For $k \in \N$, let $G$ be a $k$-step nilpotent Lie group and $\Gamma$ be a discrete, cocompact subgroup of $G$. Then $X = G/\Gamma$ is compact and $G$ acts on $X$ by left translation. For $g \in G$, the system $(X, g)$ is called a \emph{$k$-step nilsystem}. Furthermore, if $F \in C(X)$, $x \in X$, the sequence $(F(g^n \cdot x))_{n \in \N}$ is called a \emph{$k$-step nilsequence}. The family of $k$-step nilsequences forms a sub-algebra of $\ell^{\infty}$ and is closed under complex conjugation.  
    
    A sequence $(\psi(n))_{n \in \N}$ is called a \emph{uniform limit of $k$-step nilsequences} if for every $\epsilon > 0$, there exists a $k$-step nilsequence $(\psi_{\epsilon}(n))_{n \in \N}$ such that $|\psi(n) - \psi_{\epsilon}(n)| < \epsilon$ for all $n \in \N$.
    
    All $1$-step nilsequences are trigonometric polynomials (sequences having the form $(\sum_{j=1}^M c_j e^{2 \pi i n \alpha})_{n \in \N}$ for some $c_j \in \C$ and $\alpha_j \in \T$) or uniform limits of sequences of this form. On the other hand, for every $\alpha \in \T$, $(e^{2 \pi i n^2 \alpha})_{n \in \N}$ is a $2$-step nilsequence. It follows that if $(\theta(n))_{n \in \N}$ is a uniform limit of $1$-step nilsequences, $(\theta(n^2))_{n \in \N}$ is a uniform limit of $2$-step nilsequences. See \cite[Section 4.3.1]{Bergelson_Host_Kra05} or \cite[Section 11.3.2]{host_kra_18} for more details on nilsequences. 
    
    The definition of nilsequences we use here follows \cite{host_kra_18, Frantzikinakis_Host_18}. There are some slightly different definitions in the literature. For example, in \cite{Bergelson_Host_Kra05}, our $k$-step nilsequences are called \emph{basic $k$-step nilsequences}, while they define nilsequences to be uniform limits of basic nilsequences. In \cite{Green_Tao10, Green_Tao12, Green_Tao_Ziegler12}, for the sequence $(F(g^n \cdot x))_{n \in \N}$ to be called a nilsequence, the function $F$ is required to be Lipschitz instead just being continuous. 
    
\subsection{Almost periodic sequences}
\label{sec:almost_periodic}
    A \emph{(Bohr) almost periodic sequence} is a uniform limit of $1$-step nilsequences which has several characterizations. More specifically, for a bounded sequence $(\psi(n))_{n \in \N}$, the followings are equivalent:
    \begin{enumerate}
        \item $(\psi(n))_{n \in \N}$ is an almost periodic sequence.
        
        \item There exists a compact abelian group $G$, an element $g \in G$ and a continous function $F$ on $G$ such that $\psi(n) = F(g^n)$ for all $n \in \N$. 
        
        \item For every $\epsilon > 0$, the set $\{T \in \N: |\psi(n + T) - \psi(n)| < \epsilon \,\, \forall n \in \N\}$ is syndetic (i.e. has bounded gaps).
        
        \item $(\psi(n))_{n \in \N}$ is a uniform limit of trigonometric polynomials.
        
        \item $(\psi(n))_{n \in \N}$ is a uniform limit of sequence of the form $(F(n \alpha))_{n \in \N}$ where $\alpha \in \T^d$, some finite dimenional torus, and $F$ is a continuous function on $\T^d$.
        
        \item The orbit of $(\psi(n))_{n \in \N}$ under the left shift $\sigma((\psi(n))_{n \in \N}) = (\psi(n+1))_{n \in \N}$ is pre-compact under the $\ell^{\infty}$-norm. 
    \end{enumerate}
    The equivalence of above definitions can be found in \cite{petersen_1983}. 

\subsection{Sets of Bohr recurrence}
\label{sec:sets_of_bohr_recurrence}
    A set $R \subset \N$ is called a \emph{set of Bohr recurrence} if for every element $\alpha$ in a finite dimensional torus $\T^d$, the closure of $\{r \alpha: r \in R\}$ in $\T^d$ contains $0$. It is easy to show that if we remove finitely many elements from a set of Bohr recurrence, it is still a set of Bohr recurrence.
    
    By the pigeonhole principle, the set $k \N = \{kn: n \in \N\}$ is a set of Bohr recurrence. Similarly a set containing arbitrarily long arithmetic progressions of the form $\{b, 2b, \ldots, kb\}$ is a set of Bohr recurrence.
    
    Related notions are sets of topological recurrence and sets of measurable recurrence (see \cite{Furstenberg81a, Frantzikinakis_McCutcheon_2011} for definition). By definition, sets of measurable recurrence are of topological recurrence and sets of topological recurrence are of Bohr recurrence. Examples of sets of measurable recurrence are $\{P(n): n \in \N\}$ where $P \in \Z[n]$ non-constant and having zero constant term, or $\mathbb{P} - 1 = \{p - 1: p \mbox{ prime}\}$ (see \cite{Furstenberg81a, Sarkozy78, Sarkozy78b}). 
    
\subsection{Sequence \texorpdfstring{$(2^n)_{n \in \N}$}{2^n}}
\label{subsec:2^n}

\subsubsection{Nilfactors are not characteristic for multiple ergodic averages along \texorpdfstring{$(2^n)_{n \in \N}$}{2^n}}

We refer readers to \cite{Host_Kra05, Ziegler07} for definitions of characteristic factors and nilfactors. Here we show that nilfactors are not characteristic for multiple ergodic averages along $(2^n)_{n \in \N}$. 

An increasing sequence of natural numbers $(r_n)_{n \in \N}$ is called a \emph{sequence of rigidity} for a system $(X, \mathcal{B}, \mu, T)$ if for all $f \in L^2(\mu)$, $\lVert f \circ T^{r_n} - f \rVert_{L^2(\mu)} \to 0$ as $n \to \infty$. In particular, for every $A \in \mathcal{B}$, $\mu(A \cap T^{-r_n} A) \to \mu(A)$. In \cite{bergelson_deljunco_lemanczyk_rosenblatt_2014}, it is shown that $(2^n)_{n \in \N}$ is a sequence of rigidity for some non-trivial weakly mixing system $(X, \mathcal{B}, \mu, T)$. Let $A \in \mathcal{B}$ with $0 < \mu(A) < 1$, then $\mu(A \cap T^{-2^n} A) \to \mu(A)$. On the other hand, all nilfactors of a weakly mixing system are trivial \cite{Furstenberg77}. Hence the projection of $\mu(A \cap T^{-2^n} A)$ to the nilfactors is the contant $\mu(A)^2$. Because $0 < \mu(A) < 1$, $\mu(A) \neq \mu(A)^2$, so our claim follows.

\subsubsection{Some irrational rotation along $(2^n)_{n \in \N}$ is not equidistributed on $\T$}
Take $\alpha = \sum_{n \in \N} 1/2^{n^2}$. Then it is obvious that $\alpha$ is irrational and the sequence $(2^n \alpha \mod 1)_{n \in \N} $ never visits the interval $[3/4,1)$, hence not even dense in $\T$.

\subsection{Interpolation sets}
Let $G$ be a compact abelian group, $\Gamma$ its discrete dual group, i.e. the group of continuous characters on $G$. A subset $E$ of $\Gamma$ is called an \emph{$I_0$ set} (or \emph{interpolation set}) if every bounded function on $E$ is the restriction of the Fourier-Stieltjes transform of a discrete measure on $G$. An important example is when $G = \T$ and $\Gamma = \Z$, and in this paper, we only restrict to this example. For interested readers, see \cite{Graham_Hare_2013} for more information on general interpolation sets.  

\section{Connection with sets of Bohr recurrence}
\label{sec:connection_with_bohr}

\subsection{A necessary and sufficient condition}
\label{sec:necessary_sufficient}

\begin{definition*}
    For $A \subset \N$ and $\alpha \in \T^d$, define $\overline{A \alpha}$ to be the closure of $A \alpha := \{a \alpha: a \in A\}$ in $\T^d$.

    Two sets $A, B \subset \N$ are called \emph{separable by some rotation} (or just \emph{separable} for short) if there exists a finite dimensional torus $\T^d$ and an element $\alpha \in \T^d$ such that $\overline{A \alpha} \cap \overline{B \alpha} = \emptyset$.
    
    Then we also say $A$ is separable from $B$, and vice versa. 
\end{definition*}

We prove Hartman-Ryll-Nardzewski characterization with the terminology defined above (instead of using Bohr compactification) as this proof will generalize to nilsequences later. 
\begin{theorem}[Hartman-Ryll-Nardzewski \cite{Hartman_Ryll-Nardzewski_1964}]
    \label{thm:dec-31-1}
    A set $E \subset \N$ is $I_0$ if and only if every two disjoint subsets of $E$ are separable by some rotation.
\end{theorem}
\begin{proof}
    Let $E$ be an $I_0$ set and $A, B \subset E$ disjoint. There exists an almost periodic sequence $(\psi(n))_{n \in \N}$ such that $\psi(a) = 0$ for all $a \in A$ and $\psi(b) = 1$ for all $b \in B$. Fix $0 < \epsilon < 1/4$. Since $(\psi(n))_{n \in \N}$ is almost periodic, there exists a finite dimensional torus $\T^d$, a element $\alpha \in \T^d$ and a continuous function $F \in C(\T^d)$ such that $|F(n \alpha) - \psi(n)| < \epsilon$ for all $n \in \N$. Then $F(A \alpha) \subset B(0, \epsilon)$ and $F(B \alpha) \subset B(1, \epsilon)$ where $B(c, r) := \{x \in \C: |x - c| < r\}$. Since $F$ is continuous, $F(\overline{A \alpha}) \subset \overline{B(0, \epsilon)}$ and $F(\overline{B \alpha}) \subset \overline{B(1, \epsilon)}$. Because $\overline{B(0, \epsilon)} \cap \overline{B(1, \epsilon)} = \emptyset$, we have $\overline{A \alpha} \cap \overline{B \alpha} = \emptyset$. Therefore $A$ and $B$ are separable by some rotation.
    
    Conversely, assume every two disjoint subsets of $E$ are separable by some rotation. Let $r_1 < r_2 < \ldots$ be the enumeration of elements of $E$ in the increasing order and $(b(n))_{n \in \N}$ be an arbitrary bounded sequence. Without loss of generality, assume $b(n)$ takes real values and $0 \leq b(n) \leq 1$ for all $n \in \N$. It suffices to find an almost periodic sequence $(\psi(n))_{n \in \N}$ such that $\psi(r_n) = b(n)$ for all $n \in \N$.
    
    Decompose $(b(n))_{n \in \N}$ into a sum of sequences that take on only two values:
    \begin{equation*}
        b(n) = \sum_{k = 1}^{\infty} b_k(n),
    \end{equation*}
    where $b_k(n) \in \{0, 1/2^k\}$ for all $k, n \in \N$. For every $k \in \N$, let $A = \{r_n: b_k(n) = 0\}$ and $B =\{r_n: b_k(n) = 1/2^k\}$. Then $A, B$ are two disjoint subsets of $E$. According to our assumption, there exists a torus $\T^d$ and $\alpha \in \T^d$ such that $\overline{A \alpha} \cap \overline{B \alpha} = \emptyset$. By Urysohn's Lemma, there exists a continuous function $F: \T^d \to \R$ such that $F(\overline{A \alpha}) = \{0\}$ and $F(\overline{B \alpha}) = \{1/2^k\}$ while $0 \leq F \leq 1/2^k$ everywhere else. Define an almost periodic sequence $\psi_k(n) = F(n \alpha)$. By construction, $0 \leq \psi_k(n) \leq 1/2^k$ and $\psi_k(r_n) = b_k(n)$ for all $k, n \in \N$. 
    
    Then the sequence $\psi(n) := \sum_{k = 1}^{\infty} \psi_k(n)$ is a uniform limit of almost periodic sequences. Therefore it is again an almost periodic sequence and by construction, $\psi(r_n) = b(n)$ for all $n \in \N$. The proof finishes. 
    \end{proof}

\begin{lemma}
\label{prop:dec-31-2}
    The sets $A, B \subset \N$ are separable by some rotation if and only if $A - B = \{a - b : a \in A, b \in B\}$ is not a set of Bohr recurrence.    
\end{lemma}
\begin{proof}
    Assume $A$ and $B$ are separable by some rotation. Then there exists $\alpha \in \T^d$ such that $\overline{A \alpha} \cap \overline{B \alpha} = \emptyset$. This implies $0 \not \in \overline{(A - B) \alpha}$. Hence $A - B$ is not a set of Bohr recurrence.
    
    The other direction is similarly obvious.
\end{proof}

An \emph{$AP$-rich set} is a subset of $\N$ that contains arbitrarily long arithmetic progressions. 

\begin{corollary}
\label{cor:ap_rich}
    AP-rich sets are not $I_0$.
\end{corollary}
\begin{proof}
    Let $E$ be an AP-rich set. Then for every $k \in \N$, there exists $x_k, y_k \in \N$ such that $x_k, x_k + y_k, \ldots, x_k + (k-1)y_k \in E$. Let $A = \{x_k: k \in \N\}$ and $B = \bigcup_{k=1}^{\infty}\{x_k + y_k, x_k + 2 y_k, \ldots, x_k + (k-1) y_k\}$. We can choose $\{(x_k, y_k): k \in \N\}$ so that $A$ and $B$ are disjoint. Then $B - A \supset \bigcup_{k=1}^{\infty} \{y_k, 2 y_k, \ldots, (k-1) y_k\}$, which is a set of Bohr recurrence (see \cref{sec:sets_of_bohr_recurrence}). Therefore $E$ is not $I_0$ by \cref{prop:dec-31-2} and \cref{thm:dec-31-1}.
\end{proof}

\subsection{Two lemmas on unions of \texorpdfstring{$I_0$}{I0} sets}
\label{sec:some_properties}
\begin{lemma}
\label{lem:dec-31-3}
    Let $A, B, C \subset \N$. Suppose $A$ is separable by some rotation from $B$ and from $C$. Then $A$ is separable by some rotation from $B \cup C$. 
\end{lemma}
\begin{proof}
    There exists torus $\T^b, \T^c$ and elements $\beta \in \T^b, \gamma \in \T^c$ such that $\overline{A \beta} \cap \overline{B \beta} = \emptyset$ and $\overline{A \gamma} \cap \overline{C \gamma} = \emptyset$. Then $\overline{A (\beta, \gamma)} \cap \overline{B \cup C (\beta, \gamma)} = \emptyset$ where the closures are taken in $\T^{b+c}$.
\end{proof}

\begin{lemma}[{see also Graham-Hare \cite[Corollary 3.4.3]{Graham_Hare_2013}}]
\label{lem:dec-31-4}
    If $E$ and $F$ are $I_0$ and separable by some rotation, then $E \cup F$ is $I_0$.
\end{lemma}
\begin{proof}
    Let $A, B$ be two disjoint subsets of $E \cup F$. Then $A = (A \cap E) \cup (A \cap F)$ and $B = (B \cap E) \cup (B \cap F)$. The sets $A \cap E$ and $B \cap E$ are separable because they are disjoint subsets of an $I_0$ set $E$. On the other hand, $A \cap F$ and $B \cap E$ are separable because $E$ and $F$ are separable. By \cref{lem:dec-31-3}, the sets $A$ and $B \cap E$ are separable. Similarly, $A$ and $B \cap F$ are separable. Again by \cref{lem:dec-31-3}, $A$ and $B$ are separable. Since $A$ and $B$ are two arbitrary disjoint subsets of $E \cup F$, we get $E \cup F$ is $I_0$ by \cref{thm:dec-31-1}.
\end{proof}

As an application of \cref{lem:dec-31-4}, we give a proof of \cref{example:may-9-1}.
\begin{customexample}{\ref{example:may-9-1}}
    The set $E = \{2^n\}_{n \in \N} \cup \{2^n + 2n - 1\}_{n \in \N}$ is $I_0$ but $F = \{2^n\}_{n \in \N} \cup \{2^n + 2n\}_{n \in \N}$ is not $I_0$.
\end{customexample}
\begin{proof}
    Let $A = \{2^n\}_{n \in \N}$ and $B = \{2^n + 2n - 1\}_{n \in \N}$. By Strzelecki's Theorem, both $A$ and $B$ are $I_0$ since they are lacunary. On the other hand, they are separable by the rotation $1/2 \in \T$. Hence by \cref{lem:dec-31-4}, the set $E = A \cup B$ is $I_0$.
    
    However, for the set $F$, one has $\{2^n + 2n\}_{n \in \N} - \{2^n\}_{n \in \N}$ contains $2 \N$, which is a set of Bohr recurrence. Hence $F$ is not $I_0$ by \cref{prop:dec-31-2} and \cref{thm:dec-31-1}. 
\end{proof}

\subsection{Union of an \texorpdfstring{$I_0$}{I0} set and a finite set}
\label{sec:union_with_finite_set}

\begin{lemma}
    \label{lem:nov-29-1}
    Let $R$ be a set of Bohr recurrence. Then for any $\epsilon > 0$ and torus $\T^d$, there exists $N = N(R, \T^d, \epsilon)$ such that for every $\alpha \in \T^d$, there exists $r \in R \cap [N]$ satisfying $\lVert r \alpha \rVert_{\T^d} < \epsilon$ where $\lVert \cdot \rVert_{\T^d}$ denotes the distance to $0 \in \T^d$ in the flat torus metric.
\end{lemma}
\begin{proof}
    By contradiction, assume otherwise. Then there exists $d \in \N$, $\epsilon > 0$ such that for all $N \in \N$, there exists $\alpha_N \in \T^d$ satisfying $\lVert r \alpha_N \rVert_{\T^d} \geq \epsilon$ for every $r \in R \cap [N]$. Let $\alpha$ be an accumulation point of $\{\alpha_N: N \in \N\}$ in $\T^d$. Since $\T^d$ is compact, such $\alpha$ exists. Then it follows that $\lVert r \alpha \rVert_{\T^d} \geq \epsilon$ for all $r \in R$. This contradicts the hypothesis that $R$ is a set of Bohr recurrence. 
\end{proof}

In order to prove the union of an $I_0$ set with a finite set is still $I_0$, we first show that a set of Bohr recurrence can be partitioned into two sets of Bohr recurrence (\cref{prop:may-8-2}). A similar approach is also taken by Ramsey \cite{Ramsey_1980}.
\begin{proof}[Proof of \cref{prop:may-8-2}]
    Enumerate the countable set $\mathcal{S} = \{(\T^d, 1/n): d, n \in \N\}$ as $\{(\T_k, \epsilon_k):k \in \N\}$. For a set of Bohr recurrence $R$ and $(\T_k, \epsilon_k) \in \mathcal{S}$, define $N(R, \T_k, \epsilon_k)$ to be the number $N$ associated to $R, \T_k, \epsilon_k$ as in \cref{lem:nov-29-1}. 
    
    Fix a set of Bohr recurrence $R$ and let $A_1 = R \cap [N(R, \T_1, \epsilon_1)]$. Since $A_1$ is finite, $R \setminus A_1$ is still a set of Bohr recurrence. Let $B_1 = (R \setminus A_1) \cap [N(R \setminus A_1, \T_1, \epsilon_1)]$ and $R_1 = R \setminus (A_1 \cup B_1)$. Inductively, for $k \geq 2$, define 
    \[
        A_{k} = R_{k-1} \cap [N(R_{k-1}, \T_k, \epsilon_k)]
    \] 
    \[
    B_k = (R_{k-1} \setminus A_k) \cap [N(R_{k-1} \setminus A_k, \T_k, \epsilon_k)]
    \]
    \[
    R_k = R_{k-1} \setminus (A_k \cup B_k)
    \]
    For all $k \in \N$, the sets $A_k, B_k$ are finite. Hence $R_k$ is still a set of Bohr recurrence. Therefore $N(R_{k-1}, \T_{k}, \epsilon_{k})$ and $N(R_{k-1} \setminus A_{k}, \T_{k}, \epsilon_{k})$ are well-defined.  
    
    Let $A = \bigcup_{k=1}^{\infty} A_k$ and $B = \bigcup_{k=1}^{\infty} B_k$. By construction, $A$ and $B$ are disjoint subsets of $R$. It remains to show they are sets of Bohr recurrence. For any $\T^d$, $\alpha \in \T^d$ and $\epsilon > 0$, there exists $k \in \N$, such that $\T_k = \T^d$ and $\epsilon_k < \epsilon$. Then by definition of $A_k$, there exists $r \in A_k \subset A$ such that $\lVert r \alpha \rVert_{\T^d} = \lVert r \alpha \rVert_{\T_k} < \epsilon_k < \epsilon$. Since $\T^d, \alpha$ and $\epsilon$ are arbitrary, $A$ is a set of Bohr recurrence. Similar argument applies to $B$.
\end{proof}

\begin{proof}[Proof of \cref{thm:may-8-3}]
    Let $E$ be an $I_0$ set and $F$ be a finite set. We want to show $E \cup F$ is still $I_0$. By induction, we can assume $F$ consists of a single element $m \in \N$. By \cref{lem:dec-31-4} and \cref{prop:dec-31-2}, it suffices to show $E - \{m\}$ is not a set of Bohr recurrence. By contradiction, assume otherwise. According to \cref{prop:may-8-2}, the set $E-\{m\}$ can be partitioned as $A \cup B$ where $A$ and $B$ are two sets of Bohr recurrence. Hence $(A + \{m\}) - (B + \{m\}) = A - B$ is a set of Bohr recurrence. But $E = (A + \{m\}) \cup (B + \{m\})$, thus $E$ is not $I_0$ by \cref{prop:dec-31-2} and \cref{thm:dec-31-1}, a contradiction.
\end{proof}

Following result is also a corollary of \cref{prop:may-8-2}.

\begin{corollary}
\label{cor:bohr_restricted}
    Sets of Bohr recurrence are not $I_0$.
\end{corollary}
\begin{proof}
    Let $R$ be a set of Bohr recurrence. By \cref{prop:may-8-2}, there exists a partition $R = A \cup B$ where both $A$ and $B$ are sets of Bohr recurrence. Therefore $A - B$ is a set of Bohr recurrence, and thus $R$ is not $I_0$ by \cref{prop:dec-31-2} and \cref{thm:dec-31-1}.
\end{proof}

\section{Sets that are denser than lacunaries}
\label{sec:denser_than_lacunaries}

\begin{definition}
\label{def:dec-26-1}
    For a finite dimensional torus $\T^d$, define the distance between $x = (x_1, \ldots, x_d)$ and $y = (y_1, \ldots, y_d)$ on $\T^d$ by $dist(x,y) = max \{|x_i - y_i|, 1 \leq i \leq d\}$, and the distance between two finite subsets $X, Y$ of $\T^d$ by $dist(X,Y) = \min\{dist(x,y): x \in X, y \in Y\}$.
\end{definition}
\begin{remark*}
    We use above metric only for a clearer presentation. Standard flat torus metric also works.
\end{remark*}

\begin{definition*}
    For $\epsilon > 0$ and $d \in \N$, say two finite subsets of integers $A$ and $B$ are \emph{$(\epsilon, d)$-separable} if there exists $\alpha \in \T^d$ such that $dist(A \alpha, B \alpha) \geq \epsilon$.   
\end{definition*}

\begin{lemma}
\label{lem:dec-26-1}
    Let $E = \{r_n\}_{n \in \N}$ be a set that is denser than lacunaries. Then for every $\epsilon > 0$ and $d \in \N$, there exist $A, B \subset E$ finite and disjoint that are not $(\epsilon, d)$-separable. 
\end{lemma}
\begin{proof}
    By contradiction, assume there exists $\epsilon_0 > 0, d_0 \in \N$ such that every two finite disjoint $A, B \subset E$ are $(\epsilon_0, d_0)$-separable. For $F \subset E$ finite, call a subset $A$ of $F$ to be  \emph{$F$-nice} if $A$ and $F \setminus A$ are $(\epsilon_0, d_0)$-separable. By assumption, all subsets of $F$ are $F$-nice. 
    
    Let $F$ be a subset of $E$ of the form $\{r_1, r_2, \ldots, r_N\}$ for some $N \in \N$. For an $F$-nice set $A$, there exists $\alpha \in \T^{d_0}$ such that $dist(A \alpha, (F \setminus A) \alpha) \geq \epsilon_0$. Associate that $\alpha$ to $A$ by saying \emph{$\alpha$ creates $A$}. We will count the number of $F$-nice sets created by $\alpha$ as $\alpha$ ranges over $\T^{d_0}$. 
    
    For illustrative purpose, first assume $d_0 = 1$. For $\alpha \in \T$, the set $F \alpha$ is a finite subset of $\T$. Connect two elements of $F \alpha$ if their distance is less than $\epsilon_0$. Then the number of connected components is not greater than $1/\epsilon_0$. If $A$ is an $F$-nice set created by $\alpha$ and $r \alpha \in A \alpha$, every element of the connected component containing $r \alpha$ is also in $A \alpha$. Therefore, the number of $F$-nice sets created by $\alpha$ is not greater than $2^{1/\epsilon_0}$.
    
    If one perturbs $\alpha$ a little, it does not create new collection of $F$-nice sets. To make this idea precise, call $\beta \in \T$ a \emph{critical point} if $dist(r_i \beta, r_j \beta) = dist((r_j - r_i)\beta, 0) = \epsilon_0$ for some $1\leq i < j \leq N$. For each $i < j$, there are $2(r_j - r_i)$ critical points associated to them. Hence in total, there are $\sum_{1 \leq i < j \leq N} 2 (r_j - r_i) < N(N-1) r_N$ critical points. 
    
    These critical points partition $\T$ into less than $N(N-1) r_N$ subintervals. As $\alpha$ ranges over $\T$, it will create a new collection of $F$-nice sets only if it crosses a critical point. Therefore for all $\alpha$ in a subinterval, they create the same collection of $F$-nice sets.
    Thus in total there are less than
    \begin{equation}
    \label{eq:mar-22-1}
        2^{1/\epsilon_0} N(N-1) r_N 
    \end{equation}
    $F$-nice sets. Since $E$ is denser than lacunaries, for a sufficiently large $N$,  \eqref{eq:mar-22-1} will be less than $2^N$, the number of subsets of $F$. This contradicts our assumption that every subset of $F$ is an $F$-nice set.
    
    Now assume $d_0 \in \N$ arbitrary. For $F \subset E$ finite and $\alpha \in T^{d_0}$, connect two elements of $F \alpha \subset \T^{d_0}$ if their distance is less than $\epsilon_0$, with the distance is defined in \cref{def:dec-26-1}. Divide $\T^{d_0}$ into $(1/\epsilon_0)^{d_0}$ cubes of side length $\epsilon_0$. Then each cube contains at most one connected component of $F \alpha$. Hence the number of connected components is not greater than $(1/\epsilon_0)^{d_0}$. As before, this implies each $\alpha \in \T^{d_0}$ creates at most $2^{(1/\epsilon_0)^{d_0}}$ $F$-nice sets. 
    
    Similar to the case $d_0 = 1$, if we pertubes $\alpha$ a little, it does not create new collection of $F$-nice sets. To make this statement precise, we introduce a higher dimensional analogue of critical points called critical faces.
    
    For each pair $i < j \in \{1, 2, \ldots, N\}$, the set 
    \begin{equation*}
        \{ \beta \in \T^{d_0}: d(r_i \beta, r_j \beta) = d((r_j - r_i) \beta, 0_{\T^{d_0}}) =  \epsilon_0\}
    \end{equation*}
    consists of $(r_j - r_i)^{d_0}$ empty cubes, so in total there are $\sum_{1 \leq i < j \leq N} (r_j - r_i)^{d_0}$ cubes. The faces of these cubes are called \emph{critical faces}.
    
    By a very crude estimation, the critical faces partition $\T^{d_0}$ into less than 
    \begin{equation*}
        \left( \sum_{1 \leq i < j \leq N} 2(r_j - r_i) \right)^{d_0} <(N(N-1) r_N)^{d_0}
    \end{equation*}
    regions. 
    
    As $\alpha$ ranges over $\T^{d_0}$, it only creates a new collection of $F$-nice sets if it crosses a critical face. Therefore for all $\alpha$ in a region, they create the same collection of $F$-nice sets. Hence there is less than
    \begin{equation*}
        2^{(1/\epsilon_0)^{d_0}} \times (N(N-1) r_N)^{d_0} 
    \end{equation*}
    $F$-nice sets. For $N$ sufficiently large, this number is less than $2^N$. This is again a contradiction. 
\end{proof}

\begin{proof}[Proof of \cref{thm:may-8-1}]
Let $E \subset \N$ be a set that is denser than lacunaries. By \cref{lem:dec-26-1}, there exists two disjoint finite $A_1, B_1 \subset E$ that are not $(1, 1)$-separable. Let $E_2 = E \setminus (A_1 \cup B_1)$, then $E_2$ is still denser than lacunaries. There exists disjoint finite $A_2, B_2 \subset E_2$ that are not $(1/2, 2)$-separable. In general, by induction, we get a sequence of set $(A_k)_{k \in \N}$ and $(B_k)_{k \in \N}$ pairwise disjoint and for each $k \in \N$, the sets $A_k$ and $B_k$ are not $(1/k, k)$-separable. Let $A = \bigcup_{k=1}^\infty A_k$ and $B = \bigcup_{k=1}^{\infty} B_k$. Then $A, B \subset E$ disjoint and not separable by any rotation. Hence $E$ is not $I_0$ by \cref{thm:dec-31-1}. 
\end{proof}

\section{Interpolation sets for nilsequences}
\label{sec:nilsequences}

\subsection{A necessary and sufficient condition}
\label{sec:nessary-sufficient-nilsequences}
    \begin{definition*}
        Two sets $A, B \subset \N$ are said to be \emph{separable by some $k$-step nilrotation} if there exists a $k$-step nilsystem $(G/\Gamma, g)$ such that the closures in $G/\Gamma$ of $g^A \Gamma := \{g^a \Gamma: a \in A\}$ and $g^B \Gamma := \{g^b \Gamma: b \in b\}$ are disjoint.
    \end{definition*}
    
    The following proposition is analogous to Hartman-Ryll-Nardzewski characterization (\cref{thm:dec-31-1}); its proof is omitted because it is identical.
    \begin{proposition}
        For $k \in \N$, a set $E \subset \N$ is $k$-step-$I_0$ if and only if every two disjoint subsets of $E$ are separable by some $k$-step nilrotation.
    \end{proposition}
    
\subsection{Positively dense subsets of polynomial sets}
\label{sec:dense_subset_of_polynomial}
The goal of this section is to prove \cref{prop:jan-12-1}. 

\begin{definition*}
    Let $E = \{s_n\}_{n \in \mathbb{N}} \subset \N$ and $\phi = (I_N)_{N \in \mathbb{N}}$ be a sequence of intervals on $\mathbb{N}$ with $\lim_{N \to \infty} |I_N| = \infty$. We say the pair $\{E, \phi\}$ is \emph{good for averaging nilsequences} if for every nilsequence $(\psi(n))_{n \in \mathbb{N}}$, 
    \[
        \lim_{N \to \infty} \frac{1}{|I_N|} \sum_{n \in I_N} \psi(s_n)
    \]
    exists.
\end{definition*}

\begin{definition*}
    Let $E$ and $\phi$ be as before. Then for every $F = \{s_{n_i}\}_{i \in \N} \subset E$, the upper density of $F$ relative to $E$ along $(I_N)_{N \in \N}$ is defined as
    \[
        \bar{d}_{E, \phi}(F) := 
        \limsup_{N \to \infty} \frac{|\{n_i: i \in \N\} \cap I_N|}{|I_N|} 
    \]
\end{definition*}

We need following lemma which follows from \cite[Lemma 3.10]{moreira-richter-robertson}.
\begin{lemma}
    \label{lemma:subfolner}
    Suppose $\{E = \{s_n\}_{n \in \mathbb{N}}, \phi = (I_N)_{N \in \mathbb{N}}\}$ is good for averaging nilsequences and $(f(n))_{n \in \mathbb{N}}$ be a bounded sequence. Then there exists a sub-F{\o}lner sequence $\tilde{\phi} \subset \phi$ such that
    \[
        \lim_{\substack{N \to \infty \\ I_N \in \tilde{\phi}}} \frac{1}{|I_N|} \sum_{n \in I_N} \psi(s_n) f(n)
    \]
    exists for all nilsequence $(\psi(n))_{n \in \mathbb{N}}$.
\end{lemma}

Examples of pairs $\{\{s_n\}_{n \in \N}, (I_N)_{N \in \N}\}$ that are good for averaging nilsequences are:
\begin{enumerate}
    \item $s_n = P(n)$ where $P \in \Q[n]$ taking integer values on $\Z$ and $(I_N)_{N \in \N}$ is any sequence of intervals such that $\lim_{N \to \infty} |I_N| = \infty$ (\cite{leib05}).
    
    \item $s_n = P(p_n)$ where $P \in \Q[n]$ taking integer values on $\Z$, $p_n$ is the n-th prime and $I_N = [1,N]$ (\cite{Green_Tao12}).
    
    \item $s_n = \lfloor n^c \rfloor$ for any $c > 0$ and $I_N = [1,N]$ (\cite{Frantzikinakis09}).
\end{enumerate}

If the pair $\{E, \phi\}$ is good for averaging nilsequences, then $E$ is not $k$-step-$I_0$ for any $k \in \N$ because not all bounded sequences have the averages along $\phi$ converge. Now we prove it is even true for a positively dense subset of $E$.

\begin{proposition}
    \label{prop:positive-density-polynomial}
    Suppose the pair $\{E, \phi \}$ is good for averaging nilsequences. Then every subset of $E$ that has positive relative upper density along $\phi$ is not $k$-step-$I_0$ for any $k \in \N$.
\end{proposition}
\begin{proof}
Let $F = \{s_{n_i}\}_{i \in \N}$ be a subset of $E = \{s_n\}_{n \in \N}$ that has positive relative upper density along $\phi = (I_N)_{N \in \N}$. Let $A = \{n_i: i \in \N\}$ and $f(n) = 1_A(n)$. Then by assumption, $A$ has positive density along some sub-F{\o}lner sequence $\phi'$ of $\phi$, i.e. the following limit exists and positive
\begin{equation}
\label{eq:mar-26-5}
    d_{\N, \phi'}(A) = \lim_{\substack{N \to \infty \\ I_N \in \phi'}} \frac{|I_N \cap A|}{|I_N|}
\end{equation}

By \cref{lemma:subfolner}, there exists a sub-F{\o}lner sequence $\tilde{\phi} = (J_N)_{N \in \N}$ of $\phi'$ such that the following limit exists for any nilsequence $(\psi(n))_{n \in \mathbb{N}}$
\[
    L(\psi) = \lim_{N \to \infty} \frac{1}{|J_N|} \sum_{n \in J_N} 1_A(n) \psi(s_n)
\]
$L(\psi)$ can be rewritten as
\begin{equation*}
    L(\psi) =
    \lim_{N \to \infty} \frac{|J_N \cap A|}{|J_N|} \frac{1}{|J_N \cap A|} \sum_{n \in J_N \cap A} \psi(s_n)
\end{equation*}
Combining with \eqref{eq:mar-26-5}, we have the following limit exists
\begin{equation}
\label{eq:mar-26-6}
    \lim_{N \to \infty} \frac{1}{|J_N \cap A|} \sum_{n \in J_N \cap A} \psi(s_n) = \lim_{N \to \infty} \frac{|J_N|}{|J_N \cap A|} L(\psi) = \frac{L(\psi)}{d_{\N, \phi'}(A)}
\end{equation}
for any nilsequence $(\psi(n))_{n \in \N}$.

For $N \in \N$, define $K_N = \{i \in \N: n_i \in J_N \cap A\}$. Because $A$ has positive density along $(J_N)_{N \in\N}$, the size of $K_N$ goes to infinity as $N \to \infty$. The limit in \eqref{eq:mar-26-6} can be rewritten as
\[
    \lim_{N \to \infty} \frac{1}{|K_N|} \sum_{i \in K_N} \psi(s_{n_i})
\]
Since not all bounded sequences have the average along $(K_N)_{N \in \N}$ converges, we have $\{s_{n_i}\}_{i \in \N}$ is not $k$-step-$I_0$ for any $k \in \N$.
\end{proof}

\cref{prop:jan-12-1} now follows from \cref{prop:positive-density-polynomial} when $E = \{P(n)\}_{n \in \N}$ and $\phi$ is the F{\o}lner sequence along which we obtain the subset of positive relative upper density.

\subsection{Example of a set which is \texorpdfstring{$2$-step-$I_0$}{2-step-I0} but not \texorpdfstring{$1$-step-$I_0$}{1-step-I0}}
\label{sec:2_step}
    \begin{customprop}{\ref{prop:apr-27-1}}
    \label{prop:mar-26-3}
        There exists a set that is $2$-step-$I_0$ but not $1$-step-$I_0$.
    \end{customprop}
    \begin{proof}
        First, we claim that if $\{r_n^2\}_{n \in \N}$ is $1$-step-$I_0$, then $\{r_n\}_{n \in \N}$ is $2$-step-$I_0$. Indeed, let $(b(n))_{n \in \N}$ be an arbitrary bounded sequence. Because $\{r_n^2\}_{n \in \N}$ is $1$-step-$I_0$, there exists an almost periodic sequence $(\theta(n))_{n \in \N}$ such that $\theta(r_n^2) = b(n)$ for all $n \in \N$. Let $\psi(n) = \theta(n^2)$, then $(\psi(n))_{n \in \N}$ is a uniform limit of $2$-step nilsequences (see \cref{sec:preliminary_nilsequences}). We have $\psi(r_n) = \theta(r_n^2) = b(n)$ for all $n \in \N$. Since $(b(n))_{n \in \N}$ is arbitrary, the claim is proved.
        
        It remains to construct a set $E = \{r_n\}_{n \in \N}$ that is not $1$-step-$I_0$ but $\{r_n^2\}_{n \in \N}$ is $1$-step-$I_0$. We choose $E$ of the form $E = \{s_n\}_{n \in \N} \cup \{s_n + n\}_{n \in \N}$ for some sufficiently fast-growing lacunary set $\{s_n\}_{n \in \N}$. First, by \cref{prop:dec-31-2} and \cref{thm:dec-31-1}, the set $E$ is not $1$-step-$I_0$ because $\{s_n + n\}_{n \in \N} - \{s_n\}_{n \in \N}$ contains $\N$, a set of Bohr recurrence. 
        
        By Strzelecki's Theorem, since $A = \{s_n^2 : n \in \N\}$ and $B = \{(s_n + n)^2: n \in \N\}$ are lacunary, they are $1$-step-$I_0$. By \cref{lem:dec-31-4}, for $A \cup B$ to be $1$-step-$I_0$, it suffices that $B-A$ is not a set of Bohr recurrence. 
        
        Note that every element of $B - A$ is of the form $(s_n + n)^2 - s_m^2 = 2n s_n + n^2 + s_n^2 - s_m^2$ for some $m,n \geq 1$. Let
        \[
            C = \{c_n\}_{n \in \N} \mbox{ with } c_n = 2n s_n + n^2
        \]
        and 
        \[
            D = \{d_n\}_{n \in \N} \mbox{ with } d_n = s_n^2
        \]
        
        Then every element of $B-A$ is a combination of one element of $C$ and two elements of $D$.
        Fix $0 <c < 1/2$. Our goal is to find an $\alpha \in \mathbb{T}$ such that $(B-A) \alpha \subset (1/2 - c, 1/2 + c)$. Let $\ell = c/3$, then this goal will be obtained if we find an $\alpha$ such that for all $n \in \N$,
        \begin{equation}
        \label{eq:mar-26-1}
            c_n \alpha \in I_C = (1/2, 1/2 + \ell)  \mbox{ and } d_n \alpha \in I_D = (0, \ell)
        \end{equation}
        
        Now choose the sequence $(s_n)_{n \in \N}$ growing fast enough such that for all $n \in \N$,
        \begin{equation}
        \label{eq:mar-26-2}
             \frac{d_n}{c_n} = \frac{s_n^2}{2n s_n + n^2} > \frac{1}{\ell} + 1 \mbox{ and } \frac{c_{n+1}}{d_n} = \frac{2(n+1) s_{n+1} + (n+1)^2}{s_n^2}  > \frac{1}{\ell} + 1
        \end{equation}
        Then $C \cup D$ becomes a lacunary set where the ordering of its elements looks like this
        \[
            \ldots < c_n < d_n < c_{n+1} < d_{n+1} < \ldots
        \]
        For $n \in \N$, the condition that $c_n \alpha \in I_C$ is equivalent to
        \[
            \alpha \in J_{C,n} := \bigcup_{t = 0}^{c_n - 1} \left( \frac{I_C}{c_n} + \frac{t}{c_n} \right)
        \]
        Similarly, the condition that $d_n \alpha \in I_D$ is equivalent to 
        \[
            \alpha \in J_{D,n} := \bigcup_{t = 0}^{d_n - 1} \left( \frac{I_D}{d_n} + \frac{t}{d_n} \right)
        \]
        
        The set $J_{C,n}$ consists of $c_n$ subintervals of length $\ell/c_n$. By \eqref{eq:mar-26-2}, every one of these subintervals contains a subintervals of $J_{D,n}$. Similarly, each subinterval of $J_{D,n}$ contains a subinterval of $J_{C, n+1}$. Therefore $\bigcap_{n \in \N} (J_{C,n} \cap J_{D,n})$ is non-empty and any $\alpha$ in this intersection will satisfies \eqref{eq:mar-26-1}. Our proof finishes.
    \end{proof}

\section{The analogue to Frantzikinakis' question for arbitrary sequences}
\label{sec:negative_to_frantzikinakis_question}
In this section, we prove \cref{prop:negative_to_frantzikiankis_question}. We recall its statement here for convenience.

\begin{customprop}{\ref{prop:negative_to_frantzikiankis_question}}
    There exists an increasing sequence of natural numbers $(r_n)_{n \in \N}$ and a  $1$-correlation $(a(n))_{n \in \N}$ such that for any almost periodic sequence $(\psi(n))_{n \in \N}$,
    \[
        \liminf_{N \to \infty} \sum_{n=1}^N |a(r_n) - \psi(r_n)| > 0.
    \]
\end{customprop}
\begin{proof}
    First, we recall some facts about Riesz product measures and Gaussian systems. For $c_j \in \mathbb{R}$ and $|c_j| \leq 1$, define the real valued function $P_N(t) = \prod_{j = 0}^{N-1} (1 + c_j \cos (3^j t))$ on $[0,1]$ for $N \in \mathbb{N}$. Let $\sigma_N$ be the measure on $\mathbb{T}$ defined as $d \sigma_N = P_N(t) d t$. Then $\sigma_N$ converges in weak* topology to a measure $\sigma$ whose Fourier coefficients satisfy the following: For $n \in \N$, there is a unique representation of $n$ in the form $n = \sum_{j} \epsilon_j 3^j$ where $\epsilon_j \in \{0, \pm 1\}$. Then
    \[
        \hat{\sigma}(n) = \prod_{\epsilon_j \neq 0} \frac{c_j}{2}.
    \]
    The measure $\sigma$ is called a Riesz product (see \cite[pages 5-7]{Queffelec1987}). 

    For any measure $\sigma$ on $\mathbb{T}$, there exists a system, called Gaussian system, $(X, \mathcal{X}, \mu, T)$ and $f, g \in L^2(\mu)$ such that $\int_X f(x) g(T^n x) \, d \mu(x) = \hat{\sigma}(n)$ for all $n \in \N$ (see \cite[pages 369-371]{Cornfeld_Fomin_Sinai82}). By approximating $L^2$-functions by $L^{\infty}$-functions, for any $\epsilon > 0$, there exists a $1$-correlation $(a(n))_{n \in \N}$ such that $|a(n) - \hat{\sigma}(n)| < \epsilon$ for all $n \in \N$. 
    
    Let $\{r_n\}_{n \in \N}$ be the sequence $\{3^n\}_{n \in \N} \cup \{3^n + n\}_{n \in \N}$ written in the increasing order. In the construction of Riesz product measures, let $\sigma$ be the measure corresponding to the sequence $c_j = 1$ for all $j \in \N$. Then 
    \[
        \hat{\sigma}(n) = \prod_{\epsilon_j \neq 0} \frac{1}{2}
    \]
    where $n = \sum_j \epsilon_j 3^j$ with $\epsilon_j \in \{0, \pm 1\}$. Thus $\hat{\sigma}(3^n) = 1/2$ and $\hat{\sigma}(3^n + n) \leq 1/4$ for all $n \in \N$. Passing to Gaussian systems, and approximating $L^2$-functions by $L^{\infty}$ functions, there exists a $1$-correlation $(a(n))_{n \in \N}$ such that $|a(n) - \hat{\sigma}(n)| < 1/16$ for all $n \in \N$. Therefore $|a(3^n + n) - a(3^n)| \geq 1/8$ for all $n \in \N$. 
    
    Let $(\psi(n))_{n \in \N}$ be an arbitrary almost periodic sequence. The set of $n \in \N$ such that $|\psi(3^n + n) - \psi(3^n)| < 1/16$ is syndetic (see characterization (3) in \cref{sec:almost_periodic}). For $n$ in that syndetic set,
    \begin{multline*}
        |a(3^n + n) - \psi(3^n + n)| + |a(3^n) - \psi(3^n)| \geq \\
        \left| (a(3^n + n) - a(3^n)) - (\psi(3^n + n) - \psi(3^n)) \right| > 1/16
    \end{multline*}
    Therefore,
    \begin{multline*}
        \liminf_{N \to \infty} \frac{1}{2N} \sum_{n=1}^{2N} \left|a(r_n) - \psi(r_n) \right| = \\
        \liminf_{N \to  \infty} \frac{1}{N} \sum_{n=1}^{N} \left( \left| a(3^n + n) - \psi(3^n + n) \right| + |a(3^n) - \psi(3^n)|\right) > 0.
    \end{multline*}
\end{proof}

\bibliography{refs}
\bibliographystyle{plain}
\end{document}